\documentclass[12pt,reqno]{amsart}
\usepackage{amsmath,a4wide}
\usepackage{xfrac}
\usepackage{stmaryrd,mathrsfs,bm,amsthm,mathtools,yfonts,amssymb,color}
\usepackage{xcolor}
\usepackage{courier}

\begingroup
\newtheorem{theorem}{Theorem}[section]
\newtheorem{lemma}[theorem]{Lemma}
\newtheorem{proposition}[theorem]{Proposition}
\newtheorem{corollary}[theorem]{Corollary}
\endgroup

\theoremstyle{definition}
\newtheorem{definition}[theorem]{Definition}
\newtheorem{remark}[theorem]{Remark}

\numberwithin{equation}{section}
\numberwithin{subsection}{section}

\newcommand{\sing}{{\rm Sing}}

\newcommand\supp{{\rm spt}}
\newcommand\res{\mathop{\hbox{\vrule height 7pt width .3pt depth 0pt
\vrule height .3pt width 5pt depth 0pt}}\nolimits}

\newcommand{\mass}{{\mathbf{M}}}
\newcommand{\cone}{{\times\hspace{-0.6em}\times\,}}

\newcommand{\de}{\partial}
\newcommand{\bI}{{\mathbf{I}}}

\def\a#1{\left\llbracket{#1}\right\rrbracket}

\newcommand{\cH}{{\mathcal{H}}}

\newcommand{\B}{{\mathbf{B}}}
\newcommand\proj{\mathbf{p}}

\def\XXint#1#2#3{{\setbox0=\hbox{$#1{#2#3}{\int}$ }
\vcenter{\hbox{$#2#3$ }}\kern-.6\wd0}}

\newcommand\N{{\mathbb N}}

\newcommand\R{{\mathbb R}}

\newcommand{\eps}{{\varepsilon}}
\newcommand{\bA}{\mathbf{A}}

\newcommand{\cF}{{\mathcal{F}}}

\newcommand{\Lip}{{\rm {Lip}}}

\newcommand{\dist}{{\rm {dist}}}

\newcommand{\cA}{{\mathcal{A}}}

\newcommand{\bG}{{\mathbf{G}}}

\newcommand{\bM}{{\mathbf{M}}}
\newcommand{\bC}{{\mathbf{C}}}

\title[Regularity over smooth cones]{(log-)epiperimetric inequality and regularity over smooth cones for almost Area-Minimizing currents}
\author{Max Engelstein, Luca Spolaor, Bozhidar Velichkov}

\address {Max Engelstein: \newline \indent
	Massachusetts Institute of Technology (MIT), 
	\newline \indent
	77 Massachusetts Avenue, Cambridge 
	MA 02139, USA
}
\email{maxe@mit.edu}

\address {Luca Spolaor: \newline \indent
	Princeton University, Department of Mathematics, Fine Hall, 
	\newline \indent
	Washington Road, Princeton 
	NJ 08544, USA
	}
\email{lspolaor@mit.edu}

\address {Bozhidar Velichkov: \newline \indent
	Laboratoire Jean Kuntzmann (LJK), Universit\'e Grenoble Alpes
	\newline \indent
	B\^atiment IMAG, 700 Avenue Centrale, 38401 Saint-Martin-d'H\`eres}
\email{bozhidar.velichkov@univ-grenoble-alpes.fr}
\thanks{The first author was supported by an NSF MSPRF DMS-1703306. The third author has been partially supported by the projects LabEx PERSYVAL-Lab GeoSpec (ANR-11-LABX-0025-01) and ANR CoMeDiC. Finally, the authors kindly acknowledge Maria Colombo for the useful discussions and the helpful comments in the preparation of this manuscript. }

\begin{document}

\begin{abstract}
	We prove a new logarithmic epiperimetric inequality for multiplicity-one stationary cones with isolated singularity by flowing in the radial direction any given trace along appropriately chosen directions. In contrast to previous epiperimetric inequalities for minimal surfaces (e.g. \cite{Reif2}, \cite{Taylor1,Taylor2}, \cite{Wh}), we need no {\it a priori} assumptions on the structure of the cone (e.g. integrability).  If the cone is integrable (not only through rotations), we recover the classical epiperimetric inequality. As a consequence we deduce a new $\eps$-regularity result for almost area-minimizing currents at singular points where at least one blow-up is a multiplicity-one cone with isolated singularity. This result is similar to the one for stationary varifolds of L. Simon \cite{Simon0}, but independent from it since almost minimizers do not satisfy any equation. 
\end{abstract}

\maketitle

\section{Introduction}

In this paper we prove a new (log-)epiperimetric inequality for multiplicity-one smooth minimal cones. To give the precise statement, we recall the notion of spherical graph over a cone and of integrability. Let $\bC\subset \R^{n+k}$ be a multiplicity-one stationary cone and suppose that $\Sigma:=\bC\cap \de B_1$ is a smooth embedded compact submanifold of $\de B_1$. Given a function $u\in C^{1,\alpha}(\bC,\bC^\perp)$, we define its \emph{spherical graph} over $\bC$, in polar coordinates, and its \emph{renormalized volume} to be respectively 
$$
\bG_{\bC}(u):=\left\{r\frac{r\theta+u(r,\theta)}{\sqrt{r^2+|u(r,\theta)|^2}}\,:\, r\theta\in \bC \right\}
\quad\mbox{and}\quad
\cA_\bC(u):=\cH^{n}(\bG_{\bC}(u))-\cH^{n}(\bC\cap B_1)\,.
$$
Given a cone $\bC$, we say that $\bC$ is \emph{integrable} if every Jacobi field on $\bC$ is generated by a one parameter family of minimal cones; that is, if for every $1$-homogeneous solution $\phi$ of the second variation $\delta^2\cA_{\bC}(0)$, there exists a one-parameter family $(\Phi_t)_{|t|<1}$ of diffeomorphisms such that $\Phi_0=Id$, $\frac{d}{dt}\Phi_t=\phi(\Phi_t)$ and  
\begin{equation}\label{e:int}
\Phi_t(\bC)\quad\mbox{is a minimal cone with $\sing(\Phi_t(\bC))=\{0\}$ for every $|t|<1$}\,.
\end{equation} 

The (log-)epiperimetric inequality then says, roughly, that stationary cones are quantitatively isolated (as measured by $\cA_{\bC}$) in the space of cones:

\begin{theorem}[(Log-)epiperimetric inequality for multiplicity-one smooth cones]\label{t:epi_int}
	Let $\bC\subset \R^{n+k}$ be an $n$-dimensional multiplicity-one stationary cone. There exist constants $\eps, \delta>0$ and $\gamma\in [0,1)$ depending on the dimension and $\bC$ such that the following holds. Let $c \in C^{1,\alpha}(\Sigma, \bC^\perp)$ be such that $\|c\|_{C^{1,\alpha}}\leq \delta$, then there exists a function $h\in H^1(\bC\cap B_1,\bC^\perp) $ such that $h|_{\de B_1}=c$ and 
	\begin{equation}\label{e:log_epi}
	\cA_\bC(h)\leq \left(1-\eps\,|\cA_\bC(z)|^\gamma\right) \,\cA_\bC(z)	\,,
	\end{equation}
	where $z(x):=|x|\,c(x/|x|)$ is the one-homogeneous extension of $c$. If the cone $\bC$ is \emph{integrable}, then we can take $\gamma=0$. 
\end{theorem}

An epiperimetric inequality (i.e. \eqref{e:log_epi} with $\gamma = 0$) was first proven for regular points in the celebrated work of Reifenberg \cite{Reif2}, and later extended to branch points of $2$-dimensional area minimizing currents by White \cite{Wh} and to singular points of $2$-dimensional area minimizing flat chains modulo $3$ and $(\bM,\eps,\delta)$-minimizers by Taylor \cite{Taylor2,Taylor1}. In all these situations, the admissible blow-ups are cones which are \emph{integrable through rotations} (see Remark \ref{rmk:int_rot}). However there exist cones with isolated singularities which are not integrable and for which the rate of blow-up has  logarithmic decay (see \cite{Na} and Remarks 5.3 and 5.4 in \cite{AS}). Since a (classical) epiperimetric inequality implies an exponential rate of decay, we cannot hope that \eqref{e:log_epi} with $\gamma = 0$ holds for all cones. Instead we prove what is called a (log-)epiperimetric inequality, that is \eqref{e:log_epi}, with $\gamma \in [0,1)$.

 We remark that (log-)epiperimetric inequalities were introduced by the second and third named author, together with Maria Colombo, in the context of the obstacle and thin-obstacle problems \cite{CoSpVe1,CoSpVe2}; however, the proof in that setting is substantially different (and simpler). The proof of Theorem \ref{t:epi_int} bears more similarity to our recent work on isolated singularities of the Alt-Caffarelli functional, \cite{EnSpVe}; indeed this result was inspired by our work in \cite{EnSpVe}. This method seems to be very flexible and we hope to apply it to other problems (for example, Yang-Mills) and to the more difficult case of higher order singularities. 

\begin{remark}
The final steps of the proof of Theorem \ref{t:epi_int} are inspired by the beautiful work of Simon \cite{Simon0}, where the author proved uniqueness of blow-up at singularities of stationary varifolds in which at least one blow-up is a multiplicity-one cone with isolated singularity. A similar approach for generic singularities of the mean curvature flow, but with an entirely new proof of an infinite dimensional \L ojasiewicz inequality, has recently been given by Colding and Minicozzi (see \cite{CoMi}). However our approach doesn't need the surface to satisfy any PDE and is purely variational, thus allowing us to deal with almost-minimizers.
\end{remark}

\begin{remark}\label{rmk:int_rot} Recall that a cone $\bC$ is \emph{integrable through rotation}, if the family $(\Phi_t)_{|t|<1}$ in \eqref{e:int} is given by $\Phi_t = \exp (tA)$, where $A$ is any fixed $n\times n$ skew symmetric matrix. We observe that a simple modification of White's proof of the epiperimetric inequality for $2$-dimensional area minimizing cones (see \cite{Wh}) would establish an epiperimetric inequality for multiplicity-one cones with isolated singularity that are integrable through rotations. However, our proof of Theorem \ref{t:epi_int} is different than \cite{Wh, Taylor1,Taylor2}, and allows us to assume the more general notion of integrability  \eqref{e:int}, under which no epiperimetric inequality exists in the literature. In particular, this allows us to give an alternative proof of the beautiful work of Allard and Almgren \cite{AllAlm}.
\end{remark}

As a consequence of Theorem \ref{t:epi_int}, we prove a new uniqueness of the blow-up result for almost area-minimizing currents. This result is similar to the one of Leon Simon for stationary varifolds (see \cite{Simon0}), however, as mentioned above, the two results are independent from each other since stationarity and almost-minimality are independent properties. We use here standard notations for integral currents (see for instance \cite{Sim}).

\begin{definition}[Almost-Minimizers]\label{d:alm_min}
An $n$-dimensional integer rectifiable current $T$ in $\R^{n+k}$ is \emph{almost (area) minimizing} if for every $x_0\in \supp (\partial T)$ there are constants 
$C_{0}, r_0, \alpha_0 > 0$ such that 
\begin{equation}\label{e:almost minimizer2}
\|T\| (B_r(x)) \leq \|T + \partial S\|(B_r(x)) + C_{0}\, r^{n +\alpha_0}
\end{equation}
for all $0<r<r_0$ and for all integral $(n+1)$-dimensional currents $S$ supported in $B_r (x)$.
\end{definition}

\noindent For any given integer rectifiable current $R\in \bI_n (\R^{n+k})$ we define the flat norm of $R$ to be 
$$\cF (R) := \inf \{ \mass (Z) + \mass (W) : Z\in \bI_n, W\in \bI_{n+1}, Z + \partial W = R\}.$$ 

\begin{theorem}[Uniqueness of smooth tangent cone for almost minimizers]\label{t:uniq}
	Let $T\in \bI_n$ be an almost area-minimizing current and let $x_0\in\supp (T)$. Suppose that there exists a multiplicity one area minimizing cone $\bC$ such that $\bC\cap \de B_1$ is a smooth embedded orientable submanifold of $\de B_1$ and $\bC$ is a blow-up of $T$ at $x_0$. Then $\bC$ is the \emph{unique blow-up} of $T$ at $x_0$ and there exists constants $\gamma \in (0,1), C, r_0>0$, depending on $\bC$ and $n$, such that 
	\begin{gather}
	\cF((T-\bC)\res B_r)\leq C (-\log(r/r_0))^{\frac{\gamma-1}{2\gamma}}\,\qquad0<r<r_0 \,\label{e:roc1}\\
	\dist\big(\supp(T\res \B_r(x)), \bC\big)\leq C \,(-\log(r/r_0))^{\frac{\gamma-1}{2\gamma}}\,\qquad0<r<r_0 .\label{e:roc2}
	\end{gather}
	If the cone $\bC$ is integrable, then the above logarithms can be replaced by powers of $(r/r_0)$.
\end{theorem}

\noindent Similar results for almost area-minimizers are the one of Taylor \cite{Taylor1} and of the second named author together with De Lellis and Spadaro \cite{DSS1}. However there are two additional  difficulties in our situations. First of all, our epiperimetric inequality is logarithmic and not a classical one, since the cone is not assumed to be integrable. Secondly, in both \cite{Taylor1, DSS1} the admissible blow-ups are rotations of a fixed cone, so that one can assume, through a simple compactness argument, that \eqref{e:log_epi} holds at every scale. However, we do not require this to be the case, and in fact we ask for only one of the possible blow-ups to have the required structure. 

We also stress that the combined works of Allard-Almgren and Simon (e.g \cite{AllAlm,Simon0}) prove the analogous of Theorem \ref{t:uniq} for multiplicity-one stationary varifolds. However their proofs do not apply to almost minimizers as they require a PDE to be satisfied. Moreover, our approach unifies the situations of integrability and non-integrability of the cone; this relationship is investigated in Subsection \ref{ss:integrability}. 

The following corollary is a consequence of Theorem \ref{t:uniq}, since in codimension $1$ the multiplicity-one assumption on the blow-up is always guaranteed.

\begin{corollary}[Uniqueness for $7$-dimensional hypersurfaces]\label{c:codimension1}
	Suppose that  $T\in  \bI_7(U)$ is almost area-minimizing in an open set $U\subset N$, where $N$ is a $C^2$ orientable smooth manifold of dimension $8$ with $(\bar{N}\setminus N)\cap U=\emptyset$. Then $T$ has a unique tangent cone at every point and is locally $C^{1,\log}$ diffeomorphic to it.
\end{corollary}

\subsection{Idea of the proof of Theorem \ref{t:epi_int}} Let $z$ be the function of Theorem \ref{t:epi_int}, that is, the one-homogenous extension of the trace $c$. We need to construct a competitor function $h$ whose volume is smaller than that of $z$. Our first step is a slicing lemma (Lemma \ref{l:slicing}), which says that for every $g\in C^{1,\alpha}(\bC,\bC^\perp)$ we have
\begin{equation} \label{e:sl}
\cA_\bC(rg)-\cA_\bC(rc)\leq \int_0^1 \left( \cA_\Sigma(g)-\cA_{\Sigma}(c)\right)\,r^{n-1}\,dr+C\,\,\underbrace{\int_0^1 \int_{\Sigma} |\de_rg|^2\,d\cH^{n-1}\,r^{n+1}\,dr}_{=:E_r}\,,
\end{equation}
where $\cA_{\Sigma}$ is the renormalized area on the sphere defined in \ref{e:sph_area}.
In order to gain in the first term, we build $h$ by ``flowing" $c$ along $r$ so that the area of its spherical slices is decreasing. To choose good directions for the flow we use the Jacobi operator for $\cA_\Sigma$, which we denote by $\delta^2\cA_{\Sigma}$. This is an operator with compact resolvent, therefore we can decompose $c$ as
$$
c=c_{K}+c_++c_-\,,
$$ 
where $c_K$ is the projection of $c$ on the kernel of $\delta^2\cA_\Sigma$, $c_-$ is the projection on the index of $\delta^2\cA_\Sigma$ and $c_+$ is the projection on the positive eigenspaces of $\delta^2\cA_\Sigma$. Since $\Sigma$ is stationary in the sphere (being the trace of a stationary cone), the positive directions increase the volume of $\Sigma$ at second order, and so we want to move $c$ towards zero in these directions, while the negative directions decrease it, and so we don't want to move them. In general, we cannot assume that any of $c_K, c_+, c_-$ is zero, but to better explain the argument, let us address the two opposing cases, when $c_K = 0$ and when $c_+ + c_- = 0$. 

If $c_K=0$, we define
$$
h(r,\theta):=r\eta_+(r) c_+(\theta)+rc_-(\theta)\,,
$$
for a suitably chosen function $\eta_+$, with $\eta_+'=\eps$. Then, using \eqref{e:sl}, we have
$$
\cA_\bC(h)-(1-\eps)\cA_\bC(z)\leq \left(\eps (-\lambda_++\lambda_-)+C\,\|c\|_{C^{1,\alpha}(\Sigma)}+C\,\eps^2\right)\,\|c\|^2_{H^1(\Sigma)}<0\,,
$$
where $\lambda_+>0$ is the smallest positive eigenvalue of $\delta^2\cA_{\Sigma}$ and $\lambda_-<0$ the biggest negative eigenvalue, and $\eps$ depends only on the dimension and the spectral gap, and so on $\bC$.  Note that the first term on the right hand side above comes from our choice of $\eta_+$ and the aforementioned properties of the positive and negative eigenspaces of $\delta^2\cA_{\Sigma}$. The second term on the right hand side comes from the Taylor expansion of the area, while the third bounds the radial error coming from \eqref{e:sl}.

When $\bC$ is integrable through rotations  we can take $c_K = 0$ by a simple reparametrization (using for instance the implicit function theorem as in White \cite{Wh}). In the more general setting of integrability, we can also take $c_K = 0$, but we must use a slightly more complicated Lyapunov-Schmidt reduction and the analyticity of the area functional over graphs (see Subsection \ref{ss:integrability}).

If $c=c_K$ we cannot hope to gain to second order as above. Instead, following Simon \cite{Simon0}, we consider the function $A(\mu_1,\dots,\mu_l):=\cA_\Sigma(\mu_1 \phi_1+\dots+\mu_l\phi_l)$, where $\phi_1,\dots,\phi_l$ are the Jacobi fields of $\Sigma$ and $l:=\dim \ker (\delta^2\cA_\Sigma(0))<\infty$. To decrease this quantity we let the coordinates $\mu=(\mu_1,\dots,\mu_l)$ flow according to the negative gradient flow of $A$ (that is a finite dimensional mean curvature flow) in the following way
\begin{equation}\label{e:gradient}
\begin{cases}
\displaystyle\mu'(t):=-\frac{\nabla A(\mu(t))}{|\nabla A(\mu(t))|}\\
\mu(0)=\mu^0=\mbox{coordinates of $c_K$}\,,
\end{cases}
\end{equation}
and we define
$$
h(r,\theta):=r\sum_{j=1}^\ell \mu_j(\eta(r))\phi_j(\theta)\,.
$$
Clearly the function $r\mapsto A(\mu(r))$ is decreasing, but to make it quantitative we use the \L ojasiewicz inequality to deduce that for some $\gamma\in (0,1)$ and constant $C_\bC>0$ we have
$$
A(\mu(\eta(r)))-(1-\eps)A(\mu^0)\leq -\left(C_\bC\eta(r) -\eps A(\mu^0)^{\gamma}\right)\, A(\mu^0)^{1-\gamma}\,.
$$
If we choose $\eta, \eps$ proportional to a small constant times $A(\mu^0)^{1-\gamma}$, then the gain above will be larger than the radial error caused by the flow, which according to \eqref{e:sl} is proportional to $[\eta'(r)]^2$. This in turn will imply the logarithmic epiperimetric inequality \eqref{e:log_epi}. \qed

\subsection{Organization of the paper} The paper is divided in two parts; in the first part we give the proof of Theorem \ref{t:epi_int}. In the second, we show how to use Theorem \ref{t:epi_int} to deduce Theorem \ref{t:uniq}. Finally, in the appendix (for the sake of completeness), we construct the Lyapunov-Schmidt reduction. Let us point out that the proof of Theorem \ref{t:epi_int} requires no familiarity with the language of currents. However, when we apply the epiperimetric inequality to obtain regularity, we will use some theorems and notations which are standard in the literature. For an introduction to currents and their relevant properties, see \cite{Sim}.

\section{The (Log-)epiperimetric inequality via deformations along positive directions and gradient flow}

In this section we first recall some basic notations and facts about the area functional for spherical graphs. After that we give the proof of Theorem \ref{t:epi_int}.

\subsection{Preliminaries} Let $\Sigma:=\bC\cap \de B_1$ be a smooth embedded submanifold of $\de B_1$. Given a function $u\in C^1(\Sigma,\bC^\perp)$, we define its spherical graph over $\Sigma$ and its (renormalized) volume to be respectively 
\begin{equation}\label{e:sph_area}
\bG_{\Sigma}(u):=\left\{\frac{\theta+u(\theta)}{\sqrt{1+|u(\theta)|^2}}\,:\, \theta\in \Sigma \right\}
\quad\mbox{and}\quad
\cA_\Sigma(u):=\cH^{n-1}(\bG_{\Sigma}(u))-\cH^{n-1}(\Sigma).
\end{equation}

Next we recall some properties of the Jacobi operators of the area functional. The proofs of these facts are standard and can be found in \cite{Simon0}.

\begin{lemma}[First and second variations of area]\label{l:variations} Let $\Sigma$ be the spherical cross section of a stationary cone $\bC$. Then the following properties hold.
\begin{itemize}
	\item[(i)] $\cA_{\Sigma}(0)=0=\delta\cA_\Sigma(0)$ and, for every $\zeta\in C^2(\Sigma,\bC^\perp)$,
	\begin{equation}\label{e:second_var}
	\delta^2\cA_{\Sigma}(0)[\zeta,-]:=-(\Delta_\Sigma \zeta)^\perp-\sum_{i,j=1}^{n-1}(B(\tau_i,\tau_j)\cdot \zeta)\,B(\tau_i,\tau_j) -(n-1) \zeta\,
	\end{equation}
	where $(\Delta_\Sigma \zeta)^\perp$ is the projection of $\Delta_\Sigma \zeta$ on the normal bundle of $\Sigma$ in the sphere and $B$ is the second fundamental form of $\Sigma$.
	\item[(ii)] If $g\in C^{1,\alpha}(\Sigma,\bC^\perp)$, then
	\begin{equation}\label{e:continuity}
	\left|\delta^2 \cA_\Sigma(g)[\zeta,\zeta]-\delta^2 \cA_\Sigma(0)[\zeta,\zeta]\right|\leq C\, \|g\|_{C^{1,\alpha}}\,\|\zeta\|^2_{H^1}\qquad \forall \zeta \in H^1(\Sigma, \bC^\perp)\,.
	\end{equation}
\end{itemize}	
\end{lemma}

\subsection{Slicing Lemma}

In this section we estimate the difference between the area of a general graph and a cone, by bounding the additional radial error. Although simple, this lemma is the starting point of our proof, as it suggests how to modify the trace.  

\begin{lemma}[Slicing Lemma]\label{l:slicing} For every function $g = g(r,\theta) \in C^{1,\alpha}(\bC, \bC^\perp)$ the following formula holds
	\begin{gather}
	\cA_\bC(rg)\leq \int_0^1  \cA_\Sigma\left(g(r,\cdot)\right)\,r^{n-1}\,dr+C\,\left(1+\sup_{r \in (0,1)}\|g(r,\cdot)\|_{C^{1,\alpha}(\Sigma,\bC^{\perp})}\right)\,\int_0^1 \int_{\Sigma} |\de_rg|^2\,d\cH^{n-1}\,r^{n+1}\,dr\,.
	\end{gather}
	In particular, if $g(r,\theta)=c(\theta)$, then we have
	\begin{equation}\label{e:area_cono}
	\cA_{\bC}(rc)=\frac1n \cA_\Sigma(c)\,.
	\end{equation}
\end{lemma} 

\begin{proof} Consider the function $\displaystyle G(r,\theta):=r\frac{\theta+g(r,\theta)}{\sqrt{1+g^2(r,\theta)}}$. We can compute
	$$
	\cA_\bC(rg):=\int_{0}^1\int_{\Sigma}  \left|D_rG\wedge \frac1rD_\theta G \right|\,d\theta\,r^{n-1}\,dr-\cH^{n}(\bC\cap B_1)\,.
	$$
	In particular, notice that if $g(r,\theta)=c(\theta)$, then we have $|G|=r$, so that
	$$
	1=D_r |G|=\frac{G}{|G|} \cdot D_rG
	\qquad\mbox{and}\qquad|D_r G|=1\,.
	$$
	Using again $|G|=r$ and the first equality above, which implies that $D_r G=\frac{G}{|G|}$, we deduce that
	$$
	0=D_\theta |G|=\frac{G}{|G|}\cdot D_\theta G=D_rG\cdot D_\theta G\,,
	$$
	so that $\left|D_rG\wedge D_\theta G \right|=\left|D_rG\right|\,\left|D_\theta G \right|=\left|D_\theta G \right|$. From this and the fact that $r^{-1} D_\theta G$ is independent of $r$, we deduce the well known formula
	$$
	\cA_\bC(rc)=\frac1n \int_{\Sigma}\frac1r\left|D_\theta G \right|\,d\theta-\frac1n \cH^{n-1}(\Sigma)=\frac1n\left( \cH^{n-1}(\bG_\Sigma(c))- \cH^{n-1}(\Sigma)\right) =\frac1n \cA_\Sigma(c)\,.
	$$
	Now assume $g$ has no special structure; we can estimate,
	\begin{equation}\label{e:pitagora}
		\cA_\bC(rg)\leq \int_{0}^1\int_{\Sigma}  \left|D_rG\right|\,\left| \frac1rD_\theta G \right|\,d\theta\,r^{n-1}\,dr-\int_0^1\cH^{n-1}(\Sigma)\,r^{n-1}\,dr.
	\end{equation}
	A simple computation gives
	$$
	D_r G=\frac{1}{(1+|g|^2)^{\sfrac32}}\,\left(  \theta (1+|g|^2-r\,g\cdot \de_rg)+g(1+|g|^2-r\,g\cdot \de_rg) +r\,\de_rg(1+|g|^2)  \right)
	$$
	so that, using the orthogonality between $\theta$ and $g, \de_r g$, where the second follows by the fact that $\bC$ is a cone, we deduce
	\begin{align}
	|D_r G|
		&=\sqrt{1+\frac{(r\,\de_rg)^2}{1+|g|^2}} \leq 1+r^2\, (\de_rg)^2\,.\notag
	\end{align}
	Using this bound in \eqref{e:pitagora}, together with $r^{-1}|D_\theta G|\leq C\,(1+\|g\|_{C^{1,\alpha}})$, concludes the proof.
\end{proof}

\subsection{Proof of Theorem \ref{t:epi_int}} 

 We begin by constructing the competitor function $h\in H^1(\bC,\bC^\perp)\cap C^{1,\alpha}(\Sigma,\bC^\perp)$. Let $K=\ker \delta^2\cA_{\Sigma}(0)\subset L^2(\Sigma, \bC^\perp)$, where the second variation of $\cA_{\Sigma}(0)$ is the self adjoint operator with compact resolvent defined by
$$
\delta^2\cA_{\Sigma}(0)[\zeta,\cdot]:=-(\Delta_\Sigma \zeta)^\perp-\sum_{i,j=1}^{n-1}(B(\tau_i,\tau_j)\cdot \zeta)\,B(\tau_i,\tau_j) -(n-1) \zeta\,,\quad \mbox{for every }\zeta\in C^2(\Sigma,\bC^\perp)\,.
$$  
This is a system of equations, with as many equations as the dimension of the normal bundle of $\Sigma$ in the sphere. Let $\Upsilon\in C^\omega(K,K^\perp)$ be the operator given by the Lyapunov-Schmidt reduction in Appendix \ref{a:LS}, and write the trace $c$ as
$$
c=P_Kc+P_{K^\perp}c=P_Kc+\Upsilon(P_Kc)+\left(P_{K^\perp}c-\Upsilon(P_Kc)\right)\equiv P_Kc+\Upsilon(P_Kc)+c_\Upsilon^\perp\,,
$$
where $P_K$ and $P_{K^\perp}$ are the projections respectively on $K$ and $K^\perp$.
By the spectral theory for operators with compact resolvent, we know that there exists an orthonormal basis $\{\phi_j\}_{j=1}^{\infty}$ of $H^1(\Sigma,\bC^\perp)$ and numbers $\{\lambda_j\}_{j=1}^\infty$ accumulating at $+\infty$, such that
$$
\delta^2\cA_{\Sigma}(0)[\phi_j,\cdot]=\lambda_j \phi_j\,,\qquad \mbox{for every }j\in\N\,,
$$
where each eigenvalue has finite multiplicity. In particular we set $\ell:=\dim K$ and suppose that $K$ is spanned by the eigenfunctions $\phi_1,\dots,\phi_l$. Then we can decompose, 
$$
c_\Upsilon^\perp:=\sum_{\{j,\mid \lambda_j<0\}} c_j \phi_j+\sum_{\{j\mid\lambda_j>0\}} c_j \phi_j=:c_-^\perp+c_+^\perp
\qquad\mbox{and}\qquad
P_K(c):=\sum_{\{j\mid\lambda_j=0\}}\mu_j^0\,\phi_j=\sum_{j=1}^\ell \mu^0_j\,\phi_j.
$$
We then define the competitor function $h$ as
\begin{equation}\label{e:competitor}
rh(r\,\theta):=r\left(\sum_{j=1}^\ell \mu_j(\eta(r))\phi_j(\theta)+\Upsilon\left(\sum_{j=1}^\ell\mu_j(\eta(r))\,\phi_j(\theta)\right)+c_-^\perp(\theta)+\eta_+(r)\,c_+^\perp(\theta)\right)\,,
\end{equation}
where $\mu(\eta(r)):=(\mu_1(\eta(r)),\dots,\mu_\ell(\eta(r)))$ is the vector field defined by the renormalized gradient flow
\begin{equation}\label{e:gradient_flow}
\begin{cases}
\displaystyle\mu'(t):=-\frac{\nabla A(\mu(t))}{|\nabla A(\mu(t))|}\\
\mu(0)=(\mu_1^0,\dots,\mu_l^0)=:\mu^0\,,
\end{cases}
\end{equation}
where $A(\mu):=\cA_{\Sigma}\left(\sum_{\{j\mid\lambda_j=0\}}\mu_j\,\phi_j+\Upsilon \big(\sum_{\{j\mid\lambda_j=0\}}\mu_j\,\phi_j\big)\right)$ is well known to be an analytic function from $\R^\ell$ to $\R$. If $|\nabla A(\mu(t))|=0$ we set $\mu'(t)=0$.

The two cut-off functions, $\eta_+$ and $\eta$, are chosen to be
\begin{equation}\label{e:choice_eta}
\eta_+(r):=1-(1-r)\eps
\qquad \mbox{and}\qquad
\eta(r):=\eps_A A(\mu^0)^{1-\gamma}\,C\,(1-r)\,,
\end{equation}
where $\eps,\eps_A, C$ and $\gamma$ will be chosen later in the proof, depending only on $\Sigma$, and so on $\bC$.
Notice that $h(1,\cdot)=c(\cdot)$, so the first property required of our competitor is satisfied. Also note that $h \in C^{1,\alpha}(\bC, \bC^\perp)$ as each $\phi_j \in C^{1,\alpha}(\Sigma, \bC^\perp)$ (by elliptic regularity) and $\Upsilon(\mu)\in C^{1,\alpha}(\Sigma, \bC^\perp)$ (see Lemma \ref{l:LS}).
Thus we can use Lemma \ref{l:slicing}, and estimate
\begin{align}\label{e:slice}
\cA_\bC(rh)-(1-\eps)\cA_{\bC}(rc)
	&\leq \int_0^1\big(\cA_{\Sigma}(h(r, \cdot))-(1-\eps)\cA_\Sigma(c)\big)\,r^{n-1}\,dr\\
	&+\underbrace{C\,\left(1+\sup_{r \in (0,1)} \|h(r, \cdot)\|_{C^{1,\alpha}(\Sigma,\bC^{\perp})}\right)\,\int_0^1 \int_{\Sigma} |\de_rh|^2\,d\cH^{n-1}\,r^{n+1}\,dr}_{=:E_r}\,.\notag
\end{align}
By the definition of $h$ (and \eqref{e:est_upsilon}) we have that $\sup_{r \in (0,1)} \|h(r, \cdot)\|_{C^{1,\alpha}(\Sigma,\bC^\perp)}\leq 5\|c\|^\gamma_{C^{1,\alpha}(\Sigma,\bC^\perp)} \leq 1$ (for more details see \eqref{e:mucomparable} and the discussion below) and moreover
\begin{align*}
\int_0^1 \int_{\Sigma} |\de_rh|^2\,d\cH^{n-1}\,r^{n+1}\,dr
	&\leq 2\, \int_0^1 r^{n+1}\,\Big( (\eta_+'(r))^2\|c_+^\perp\|^2+(\eta'(r))^2\|P_Kc+\Upsilon(P_Kc)\|^2  \Big)\,dr\\
	&\leq C\, \int_0^1 r^{n+1}\,\left( \eps^2\|c_\Upsilon^\perp\|^2_{H^1(\Sigma,\bC^\perp)}\,+(\eta'(r))^2 \right)\,dr\,,
\end{align*}
where in the second inequality we used \eqref{e:est_upsilon} to estimate $\|P_Kc+\Upsilon(P_Kc)\|\leq 2\|P_Kc\| < 1$. It follows that
\begin{equation}\label{e:radial_error}
|E_r|\leq  C\,\left( \eps^2 \|c_\Upsilon^\perp\|_{H^1}^2 +\eps_A^2 A(\mu^0)^{2-2\gamma}  \right)\,.
\end{equation}
For the main term in the right-hand side of \eqref{e:slice}, we split the estimate in two parts
\begin{align*}
\cA_{\Sigma}(h)-(1-\eps)\cA_\Sigma(c)
	&=\underbrace{\big(\cA_{\Sigma}(h)-\cA_{\Sigma}(\mu+\Upsilon(\mu))\big)-(1-\eps)\big(\cA_\Sigma(c)-\cA_{\Sigma}(P_Kc+\Upsilon(P_Kc))\big)}_{=:E^\perp}\\
    &\quad+\underbrace{ \cA_{\Sigma}(\mu+\Upsilon(\mu))-(1-\eps)\cA_{\Sigma}(P_Kc+\Upsilon(P_Kc))  }_{=:E^T}\,.
\end{align*} 
For the first part, denoting by $h^\perp_\Upsilon:=h-(\mu+\Upsilon(\mu))$, we have by a simple Taylor expansion
\begin{align}\label{e:eperp1}
E^\perp	
	&=\delta\cA_{\Sigma}(\mu+\Upsilon(\mu))[h^\perp_\Upsilon]+\delta^2\cA_{\Sigma}(\mu+\Upsilon(\mu)+sh^\perp_\Upsilon)[h^\perp_\Upsilon,h^\perp_\Upsilon]\notag\\
	&\quad-(1-\eps) \left( \delta\cA_{\Sigma}(\mu^0+\Upsilon(\mu^0))[c^\perp_\Upsilon]+\delta^2\cA_{\Sigma}(\mu^0+\Upsilon(\mu^0)+tc^\perp_\Upsilon)[c^\perp_\Upsilon,c^\perp_\Upsilon] \right)\notag\\
	&\leq \delta^2\cA_{\Sigma}(\mu+\Upsilon(\mu)+sh^\perp_\Upsilon)[h^\perp_\Upsilon,h^\perp_\Upsilon]-(1-\eps)\, \delta^2\cA_{\Sigma}(\mu^0+\Upsilon(\mu^0)+tc^\perp_\Upsilon)[c^\perp_\Upsilon,c^\perp_\Upsilon]\,,
\end{align}
where $s,t\in  (0,1)$ and the second inequality holds thanks to \eqref{e:LSorth} and the fact that $h^\perp_\Upsilon,c^\perp_\Upsilon\in K^\perp$. Using (ii) of Lemma \ref{l:variations}, we have that 
\begin{align*}
\left|\delta^2\cA_\Sigma(f)[\zeta,\zeta]-\delta^2\cA_\Sigma(0)[\zeta,\zeta]\right|\leq C\, \|f\|_{C^{1,\alpha}(\Sigma,\bC^{\perp})}\,\|\zeta\|^2_{H^1(\Sigma,\bC^{\perp})}\,.
\end{align*}
Using this estimate in \eqref{e:eperp1} and the fact that $|\eta|, |\eta_+| \leq 1$, we deduce
\begin{align}\label{e:eperp2}
E^\perp
	&\leq \delta^2\cA_{\Sigma}(0)[c^\perp_-+\eta_+c^\perp_+,c^\perp_-+\eta_+c^\perp_+]-(1-\eps)\delta^2\cA_{\Sigma}(0)[c^\perp_-+c^\perp_+,c^\perp_-+c^\perp_+]\notag\\
	&\quad +C\,\left(\|\mu+\Upsilon(\mu)+sh_\Upsilon^\perp\|_{C^{1,\alpha}}+\|\mu^0+\Upsilon(\mu^0)+tc_\Upsilon^\perp\|_{C^{1,\alpha}}\right)\,\|c^\perp_\Upsilon\|^2_{H^1(\Sigma,\bC^{\perp})}\notag\\
	&\leq \eps\, \delta^2\cA_{\Sigma}(0)[c^\perp_-,c^\perp_-]+\left(\eta^2_+-(1-\eps)\right) \delta^2\cA_{\Sigma}(0)[c^\perp_+,c^\perp_+]\notag\\
	&\quad +C\,\left(2\|c^\perp_\Upsilon\|_{C^{1,\alpha}}+\|\mu\|_{C^{1,\alpha}}+\|\mu^0\|_{C^{1,\alpha}}\right)\,\|c^\perp_\Upsilon\|^2_{H^1(\Sigma,\bC^{\perp})}\,.
\end{align}
Integrating \eqref{e:eperp2} in $r$ and recalling that, by \eqref{e:choice_eta}, we have $\int_0^1 (\eta^2_+(r)-(1-\eps))\,r^{n-1}\,dr\leq -\eps$, we conclude
\begin{align}\label{e:eperp3}
\int_0^1E^\perp\,r^{n-1}\,dr
	&\leq \eps \, \max_{\lambda_j<0}\lambda_j\,\|c_-^\perp\|^2_{H^1(\Sigma,\bC^{\perp})}-\eps \, \min_{\lambda_j>0}\lambda_j\,\|c_-^\perp\|^2_{H^1(\Sigma,\bC^{\perp})}\notag\\ 
	&\quad +C\,\left(\|c_\Upsilon^\perp\|_{C^{1,\alpha}} + \|\mu\|_{C^{1,\alpha}}+\|\mu^0\|_{C^{1,\alpha}}\right)\,\|c^\perp_\Upsilon\|^2_{H^1(\Sigma,\bC^{\perp})}\notag\\
	&\leq - \left(C_\bC\,\eps-C\,\left(\|c_\Upsilon^\perp\|_{C^{1,\alpha}} +\|\mu\|_{C^{1,\alpha}}+\|\mu^0\|_{C^{1,\alpha}}\right)\right)\,\|c^\perp_\Upsilon\|^2_{H^1(\Sigma,\bC^{\perp})},
\end{align}
where $C_\bC>0$ is a strictly positive constant depending only on the spectral gap between $0$ and the other eigenvalues of $\delta^2\cA_{\Sigma}(0)$, that is depending only on $\Sigma$ and so, on $\bC$. Noticing that, by definition of $\eta$, we have
\begin{gather}\label{e:mucomparable}
|\mu(\eta(r))-\mu^0|\leq \int_0^{\eta(r)}|\mu'(t)|\,dt\leq |\eta(r)|\leq C\,\eps_A\,A(\mu^0)^\gamma\notag\\
\left|\frac{d}{dr}\mu(\eta(r))\right|\leq |\mu'(\eta(r))|\,|\eta'(r)|\leq C\,\eps_A\,A(\mu^0)^\gamma\notag\,.
\end{gather}
These estimates, combined with elliptic regularity, allow us to bound $\|\mu\|_{C^{1,\alpha}} \leq C\eps_A A(\mu^0)^\gamma +  \|\mu^0\|_{C^{1,\alpha}} \leq 2\|\mu^0\|_{C^{1,\alpha}}^\gamma$ (for $\eps_A > 0$ sufficiently small but depending only on $\bC$), so that by choosing $\|c\|_{C^{1,\alpha}}$ (which is bigger than $\|\mu^0\|_{C^{1,\alpha}},\|c_\Upsilon^\perp\|_{C^{1,\alpha}}$) sufficiently small in \eqref{e:eperp3}, depending only on $C_\bC$, we conclude
\begin{equation}\label{e:eperp4}
\int_0^1E^\perp\,r^{n-1}\,dr\leq - C_\bC\,\eps\,\|c^\perp_\Upsilon\|^2_{H^1(\Sigma,\bC^{\perp})}\,.
\end{equation}

Next we estimate $E^T$. Recall that, by the \L ojasiewicz inequality for the analytic function $A$ (see \cite{Loj}), there exist a neighborhood $U$ of $0$ and constants $C,\gamma>0$  depending on $\bC$ and the dimension $n$, with $\gamma\in (0,\sfrac12]$, such that
\begin{equation}\label{e:loj}
|A(\mu)|^{1-\gamma}\leq C\, |\nabla A(\mu)|\,, \qquad \mbox{ for every }\mu\in U\,.
\end{equation}
Then, as long as $A(\mu(s))>0$ for $0<s<t$, we can estimate,
\begin{align}\label{e:mon}
A(\mu(t))-A(\mu^0)=
\int_0^t\nabla A(\mu(\tau))\cdot\mu'(\tau)\,d\tau=-\int_0^t|\nabla A(\mu(\tau))|\,d\tau\leq0\,,
\end{align}
so that the function $t\mapsto A((\mu(t))$ is non increasing, and therefore there exists a first time $t_1>0$ such that
$$
\begin{cases}
A(\mu(t))\geq \frac12 A(\mu^0)>0 & \mbox{if } 0\leq t\leq t_1\\
A(\mu(t))\leq \frac12 A(\mu^0)  & \mbox{if } t\geq t_1\,.
\end{cases}
$$
If $\eta(r) \leq t_1$ then we have,
\begin{align}\label{e:flux1}
E^T
	&= A(\mu(\eta(r)))-A(\mu^0)+\eps\,A(\mu^0) 
	\leq -\int_{0}^{\eta(r)}|\nabla A(\mu(\tau))|\,d\tau +\eps A(\mu^0)\notag \\
	&\leq-C_\bC \int_{0}^{\eta(r)}|A(\mu(\tau))|^{1-\gamma}\,d\tau+\eps A(\mu^0)
	\leq -\,C_\bC \,A(\mu(\eta(r))^{1-\gamma} \,\eta(r)+\eps A(\mu^0)\notag \\
	&\leq -\frac{C_\bC}{2^{1-\gamma}}\,|A(\mu^0)|^{1-\gamma} \,\eta(r)+\eps A(\mu^0)
	= -\left(\tilde{C}_\bC\eta(r) -\eps A(\mu^0)^{\gamma}\right)\, A(\mu^0)^{1-\gamma}
\end{align}
where in the first inequality we used \eqref{e:mon}, which holds since $A(\mu(t))>0$, and in the second inequality we used the \L ojasiewicz inequality \eqref{e:loj}. Finally, in the third inequality we use the monotonicity of $A$ and in the fourth we use $\eta(r) \leq t_1$.

\noindent If $\eta(r) > t_1$, then 
\begin{align}\label{e:flux2}
E^T
	&=A(\mu(\eta(r)))-(1-\eps)A(\mu^0) 
	< -\left(\frac{1}{2}-\eps \right)A(\mu^0)\notag\\ 
	&< -\left(C_\bC\,\eta(r)-\eps A(\mu^0)^\gamma \right)\, A(\mu^0)^{1-\gamma},
\end{align}
where the last inequality holds since $|\eta|\leq C\, \eps_A\, A(\mu^0)^{1-\gamma}<\frac12$ as long as $\mu^0$ is small enough. 

We now consider two cases:

\noindent {\bf Case 1:} $A(\mu^0)^{1/2}<\,\tau \|c^\perp\|_{H^1(\Sigma,\bC^\perp)} $, for some $\tau > 0$ small but universal (i.e. depending only on $\bC$ and $n$, but not on $c$). In this scenario, let $\eta \equiv 0$ (i.e. $\eps_A \equiv 0$), so that $E^T=\eps \,A(\mu^0)$, and combine \eqref{e:slice}, \eqref{e:radial_error} and \eqref{e:eperp4}, to deduce that 
\begin{align}\label{e:epicase1}
\cA_\bC(rh) - (1-\eps)\cA_\bC(rc)
	& \leq -C_{\bC} \eps\|c^{\perp}_\Upsilon\|_{H^1}^2 + (\eps  A(\mu^0)+\eps^2 \,\|c_\Upsilon^\perp\|_{H^1}^2)\notag\\
	&\leq -\left(C_{\bC}-\tau-\eps\right)\,\eps\|c^{\perp}_\Upsilon\|_{H^1}^2 < 0\,,
\end{align} 
for a proper choice of $\eps>0$ and $\tau>0$ small enough depending only on $n$ and $\bC$.

\smallskip

\noindent {\bf Case 2:} Otherwise, we choose $\eps = \eps_A A(\mu^0)^{1-\gamma}$ for some $\eps_A > 0$ small, depending only on $n$ and $\bC$. Using \eqref{e:flux1} and \eqref{e:flux2} we can estimate
\begin{align}\label{e:flux3}
\int_0^1 E^T r^{d-1} dr 
	&\leq -A(\mu^0)^{1-\gamma}\int_0^1  \left(C_\bC\,\eta_A(r) -\eps A(\mu^0)^\gamma \right) r^{d-1}\ dr\\
	&=-\eps_A A(\mu^0)^{2-2\gamma} \int_0^1 (C_\bC C(1-r)-A(\mu^0)^\gamma)r^{d-1}dr \leq -C_\bC \eps_A A(\mu^0)^{2-2\gamma}\,.\notag
\end{align}
Then, using \eqref{e:flux3} together with \eqref{e:slice}, \eqref{e:radial_error} and \eqref{e:eperp4}, we deduce 
\begin{align}\label{e:case21}
\cA_{\bC}(rh)-(1-\eps)\,\cA_{\bC}(rc)
	& \leq \underbrace{-C_{\bC} \eps\|c^{\perp}_\Upsilon\|_{H^1}^2}_{E^\perp} \underbrace{-C_\bC \eps_A A(\mu^0)^{2-2\gamma}}_{E^T}+\underbrace{C\,\left(\eps^2 \,\|c_\Upsilon^\perp\|_{H^1}^2+\eps_A^2\,A(\mu^0)^{2-2\gamma}\right)}_{E_r}\notag\\
	&\leq -(C_{\bC} \eps-\eps^2)\,\|c^{\perp}_\Upsilon\|_{H^1}^2- (C_\bC \eps_A +C\,\eps_A^2)\,A(\mu^0)^{2-2\gamma}<0\,,
\end{align}
since we are in the case $A(\mu^0)>0$ and by choosing  $\eps_A$ small enough depending only on $n$ and $\bC$. Moreover, since we are in the case $A(\mu^0)^{1/2}\geq \tau\,\|c^\perp\|_{H^1(\Sigma,\bC^\perp)} $, we can use Lemma \ref{l:slicing} to write
\begin{align}\label{e:case22}
\cA_{\bC}(rc)
	&=\frac1n \cA_\Sigma(c)=\frac1n \left(\cA_\Sigma(c)-\cA(P_Kc+\Upsilon(P_Kc))+\cA(P_Kc+\Upsilon(P_Kc))\right)\notag\\
	&\leq C_\bC\, \|c_\Upsilon^\perp\|_{H^1}^2+A(\mu^0)\leq \left(C_{\bC}\tau^{-1} + 1\right) A(\mu(0)),
\end{align}
where in the first inequality we used Lemma \ref{l:variations} combined with the standard Taylor expansion of the area. Finally, combining \eqref{e:case21} and \eqref{e:case22} we conclude
\begin{equation}\label{e:epicase2}
\cA_{\bC}(rh)-(1-\eps_A(\cA_{\bC}(rc))^{1-\gamma})\,\cA_{\bC}(rc)<0\,.
\end{equation}

Combining the two previous cases concludes the proof.
\qed

\subsection{The integrability case}\label{ss:integrability} To finish the proof of the epiperimetric inequality for integrable cones we need the following lemma, which is based on the analyticity of $A$ and whose proof can be found also in \cite{AS}. 

\begin{lemma}[Constant area on the kernel]\label{l:intimpliescost}
	A cone $\bC$ is integrable if and only if $A(\mu)=A(0)=0$ in a neighborhood of $0$.
\end{lemma}

Using this lemma  it is immediate to see that if $\bC$ is integrable then we always fall in Case 1 of the proof of Theorem \ref{t:epi_int}, so that we have \eqref{e:log_epi} with $\gamma=0$.  

\begin{proof}[Proof of Lemma \ref{l:intimpliescost}] The integrability condition \eqref{e:int} is equivalent to
	\begin{equation}\label{e:int2}
	\forall \phi\in \ker  \delta^2\cA_{\Sigma}(0)\quad \exists (\Psi_s)_{s\in(0,1)}\subset C^2(\Sigma, \bC^\perp)\quad\mbox{s.t. }
	\begin{cases}
	\lim_{s\to 0}\Psi_s=0\\
	\delta\cA_{\Sigma}(\Psi_s)=0\quad \mbox{for }s\in(0,1)\\
	\displaystyle\frac{d}{ds}\Big|_{s=0}\Psi_s=\lim_{s\to 0}\frac{\Psi_s}{s}=\phi \,.
	\end{cases}
	\end{equation}  

Assume \eqref{e:int2} holds, and recall the definition $A(\mu) = \cA_{\Sigma}(\mu + \Upsilon(\mu))$. If $A \equiv 0$ in a neighborhood of zero then we are done. Otherwise we can write $A(\mu) = A_p(\mu) + A_R(\mu)$ where, $A_p \not\equiv 0$,  $A_p(\lambda \mu) = \lambda^p A(\mu)$ for $\lambda > 0$ and $A_R(\mu)$ is the sum of homogeneous polynomials of degrees $\geq p+1$. Note there exists some $\phi \in \ker \delta^2 \cA_{\Sigma}(0)$ such that $\nabla A_p(\phi) \neq 0$; let $\Psi_s$ be the one-parameter family of critical points that is generated by $\phi$ (as in \eqref{e:int2}). 

As $\Psi_s$ is a critical point, Lemma \ref{l:LS} allows us to write $\Psi_s = \phi_s + \Upsilon(\phi_s)$ where $\phi_s \in K$ and $\frac{\phi_s}{s} \rightarrow \phi$ as $s\downarrow 0$. Computing $$0 = \delta \mathcal A_{\Sigma}(\Psi_s) = \nabla A(\phi_s) = \nabla A_p(\phi_s) + \nabla A_R(\phi_s) = s^{p-1}\nabla A(\phi) + o(s^{p-1}).$$ Divide the above by $s^{p-1}$ and let $s\downarrow 0$ to obtain a contradiction to $\nabla A_p(\phi) \neq 0$. 

In the other direction assume that $A \equiv 0$ in a neighborhood of $0$. This implies that $\nabla A \equiv 0$ in a (perhaps slightly smaller) neighborhood of $0$. Therefore, for any $\mu \in \ker \delta^2 \cA_{\Sigma}(0)$, letting $\Psi_s = s\mu + \Upsilon(s\mu)$ and recalling \eqref{e:est_upsilon} establishes \eqref{e:int2}.
\end{proof} 

\section{Almost area minimizing currents  and applications}

In this section we apply the (log-)epiperimetric inequality of Theorem \ref{t:epi_int} to deduce Theorem \ref{t:uniq}. As mentioned in the introduction, for the classical epiperimetric inequality this has been done in \cite{DSS1} by De Lellis, Spadaro and the second author. Here, however, the strategy is slightly different since we do not know that every blowup is of the same type (i.e. our uniqueness result not only determines a rotation, but actually prevents the formation of additional singularities).

\subsection{Technical preliminaries}

%

We start by recalling the following well-known proposition, whose proof can be found in \cite[Proposition 2.1]{DSS1}.

\begin{proposition}[Almost Monotonicity {\cite[Proposition 2.1]{DSS1}}]\label{p:AMO}
	Let $T\in {\bf I}_n (\R^{n+k})$ be an almost minimizer and $x\in \supp (T)\setminus \supp (\partial T)$. There are constants $C, \bar{r}, \alpha_0>0$ such that 
	\begin{equation}\label{e:monotonicity formula}
	\int_{B_r (x)\setminus B_s(x)} \frac{|(z-x)^\perp|^2}{|z-x|^{n+2}} d \|T\|(z)
	\leq C\left( \frac{\|T\|(B_r(x))}{\omega_n\,r^n}
	-\frac{\|T\|(B_s(x))}{\omega_n\,s^n} +\,r^{\alpha_0} \right),
	\end{equation}
	for all $0<s<r<\bar{r}$ (in \eqref{e:monotonicity formula} $(z-x)^\perp$ denotes the projection of the vector 
	$z-x$ on the orthogonal complement of the approximate tangent to $T$ at $z$).
	
	 In particular, the function $\displaystyle{r\mapsto  \frac{\|T\|(B_r(x))}{\omega_n\,r^n}+r^{\alpha_0}}$ is nondecreasing.
\end{proposition} 

 Using \eqref{e:monotonicity formula} together with the almost-minimizing property, it is easy to see that the same blow-up analysis holds for almost-minimizing and minimizing currents. That is, we can consider the blow-up sequence of $T$ at $x$ defined by  $T_{x,r}:=(\iota_{x,r})_\sharp T$, where the map $\iota_{x,r}$ is given by $\R^{n+k} \ni y \mapsto \frac{y-x}{r}\in \R^{n+k}$. Recall that an area-minimizing cone $S$ is an integral area-minimizing current such that $(\iota_{0,r})_\sharp S = S$ for every $r>0$. Then, by the almost monotonicity of $\|T_{x,r}\|$,  $T_{x,r}\to S$ up to subsequences, with $S$ an area minimizing cone. Furthermore, by the almost minimality of $T$, the convergence is strong, i.e. their difference goes to zero in the flat norm, the support of $T_{x,r}$ converges to the support of $S$ in the Hausdorff distance and the mass of $T_{x,r}$ converges to that of $S$.
 
 In what follows, we continue to denote by $\bC$ an arbitrary multiplicity-one area minimizing cone; where we slightly abuse notation  and identify the cone with its support. Moreover, $T$ will be almost area minimizing with parameters $r_0, \alpha_0$ (see Definition \ref{d:alm_min}) and $T_r:=T_{0,r}$. Finally, we set
$$
\Theta_M(T,x):=\lim_{r\to 0}\frac{\|T\|(B_r(x))}{\omega_n\,r^n}
\qquad\mbox{and}\qquad
\Theta_\bC:=\|\bC\|(B_1)=\frac{\|\bC\|(B_r)}{\omega_n\,r^n}\,.
$$

 We first prove a standard parametrization lemma over a multiplicity-$1$ cone. 

\begin{proposition}[Spherical parametization from a cone]\label{p:sph_par}
	Let $\tau,\eps\in(0,\sfrac14)$, $\bC$ be a multiplicity-$1$ area minimizing cone, and $T\in \bI_n$ be an almost area minimizing current with $\Theta_M(T,0)=\Theta_{\bC}$. There are constants $\delta_1, \eta, r_1 > 0$, (which depend on $\tau, \eps$, the almost-minimizing parameters $C, \alpha_0, r_0$ and the dimension and co-dimension $n, k$) such that if $r<r_1$, and
	\begin{equation}\label{e:par_hyp}
	\frac{\|T\|( B_{4r})}{{(4r)}^n\omega_n}-\Theta_M(0)\leq \eta
	\quad\mbox{and}\quad
	\cF\big(\de((T_{2r}-\bC)\res B_{1})\big)\leq\eta\,,
	\end{equation}
	then there exists $u\in C^{1,\alpha}(\bC \cap B_{r}\setminus B_{\tau r} ,\bC^\perp)$ such that
	\begin{gather}
	T \res(B_{r}\setminus B_{\tau r})=\bG_\bC(u)
	\qquad\mbox{and}\qquad
	\|u\|_{C^{1,\alpha}}\leq \eps\label{e:par_2}\,.
	\end{gather}
\end{proposition}

\begin{proof}
	Arguing by contradiction, we assume there exist sequences of almost area minimizing currents $(T^k)_k$, all with the same constants $r_0,C,\alpha_0>0$, and radii $(r_k)_k$, with $r_k\to 0$, such that, if we consider $R_k:=(T^k)_{r_k}$, then
	\begin{equation}\label{e:cont}
	\frac{\|R_k\|(B_4)}{4^n\omega_n}-\Theta_\bC\leq \frac 1k
		\quad\mbox{and}\quad
		\cF(\de((R_k-\bC)\res B_2))\leq\frac1k\,.
	\end{equation}
	Notice that, by the first inequality above, we have a uniform bound for $\|R_k\| (B_4)$, so that up to subsequences, $R_k\to V$ in $B_4$. By the same uniform bound and the usual slicing lemma, passing to a subsequence there is a radius $\rho_0\in ]2, 4[$ such $\mass (\partial ((R_k-V)\res B_{\rho_0}))$ is uniformly bounded.
	On the other hand $R_k -V$ is converging to $0$ in the sense of currents and hence, by \cite[Theorem 31.2]{Sim}, $\mathcal{F} ((R_k-V)\res B_{\rho_0})\to 0$.
	This means, for all $\rho \leq \rho_0$, that there are integral currents $H_k, G_k$ (depending on $\rho$) with $\mass (H_k)+\mass (G_k)\to 0$ such that
	\[
	(R_k-V)\res B_{\rho} = \partial H_k + G_k\, .
	\]
	Taking the boundary of the latter identity we conclude that $\partial G_k = \partial ((R_k-V)\res B_\rho)$. Now, rescaling the almost minimality property of $T_k$,
	we conclude that
	\[
	\|R_k\| (B_\rho) \leq \|V\| (B_\rho) + \mass (G_k) + C \rho^{\alpha_0} r_k^{\alpha_0}\, .
	\]
	Since $(\mass (G_k) + r_k)\downarrow 0$, we infer 
	\[
	\limsup_{k\to\infty} \|R_k\| (B_\rho) \leq \|V\| (B_\rho)\, .
	\]
	On the other hand, $R_k\to V$ in $B_1$, so we also have
	\[
	\|V\| (B_\rho) \leq \liminf_{k\to\infty} \|R_k\| (B_\rho)\, .
	\]
	Using the almost monotonicity identity \eqref{e:monotonicity formula} and passing to the limit in $k$, we conclude by a standard argument that $V\res B_{2}$ is a cone.
    Passing to the limit in the second inequality of \eqref{e:cont} we get $\de(V\res B_2)=\de(\bC\res B_2)$. Since both $V$ and $\bC$ are integral cones, we deduce that $V=\bC$.  Finally, since $\bC$ has multiplicity $1$ and is smooth away from $0$, and $R_k$ converges to $\bC$ by Allard's theorem for almost area minimizing currents (see for instance \cite{SS}) we get a contradiction.
\end{proof}

Before we can prove Theorem \ref{t:uniq} we need to estimate the difference between tangent cones coming from comparable scales.

\begin{lemma}[Tangent cones at comparable scales]\label{l:tg_cones}
	Let $T$ be an almost area minimizing integral current. Then for all $\eps_2>0$ there exists $\delta_2=\delta_2(\eps_2, C, \alpha_0, r_0)>0$ such that for all $0<2r<\delta_2$ and all $\rho\in [r,2r]$ we have
	\begin{equation}\label{e:comp_scales}
	\cF((T_\rho-T_r)\res \de B_1)<\eps_2\,. 	
	\end{equation}
\end{lemma}

\begin{proof}
	We argue by contradiction. Assume there are sequences $r_n\downarrow0$ and $\rho_n\downarrow0$, with $\rho_n\in[\sfrac{r_n}2,r_n]$, and such that
	$$
	\cF((T_{\rho_n}-T_{r_n})\res \de B_1)\geq\eps_2\,.
	$$
	As $1\leq \frac{r_n}{\rho_n}\leq 2$ for every $n\in  \N$, we can assume (passing to subsequences) $0< L = \lim_n\frac{r_n}{\rho_n} <\infty$. We then compute
	$$
	\cF\bigl((V-\lim_{n\to\infty}T_{r_n})\res \de B_1\bigr)=\cF\bigl( \lim_{n\to\infty}(\iota_{L})_\sharp \bigl((V-T_{\rho_n})\res \de B_1\bigr) \bigr)=0
	$$
	where $V$ is the tangent cone associated with the sequence $(\rho_n)_n$ and we used the fact that $V$ is a cone. It follows that both sequences $T_{r_n}$ and $T_{\rho_n}$ approach the same tangent cone $V$, and by triangle inequality we get a contradiction for $n$ big enough.
\end{proof}

\subsection{Proof of Theorem \ref{t:uniq}} Assume that $x_0=0$. We divide the proof in several steps.

\medskip

\noindent\textbf{Step 1: (Log-)Epiperimetric inequality.} Assume that for every $0<s<r<r_0$, there exists a $c$ with small $C^{1,\alpha}$, such that $T_r\res \de B_1=\bG_{\Sigma}(c)$.  By Theorem \ref{t:epi_int} there exists $\eps, C, \gamma>0$, with $\gamma\in[0,1)$ and $h\in H^{1}(\bC,\bC^\perp)$ such that
$$
\cA_{\bC}(h)\leq (1-\eps\, \cA_{\bC}(rc)^\gamma)\cA_\bC(rc)\,.
$$
Set $f(r):=\|T\|(B_r)-\Theta_\bC\,r^n$ and recall that, since $r \mapsto \|T_r\|(B_r)$ is monotone, the function
$f$ is differentiable a.e. and its distributional derivative is a measure. Its absolutely continuous part
coincides a.e. with the classical differential and its singular part is nonnegative, so that
$$
r^n\cA_{\bC}(rc)=\|0\cone (T\res \de B_r)\|(B_r)-\Theta_\bC\,r^n\leq \frac{r}{n}\,f'(r)\,.
$$ 
Using the almost minimality of $T$ and the previous two inequalities, we get
\begin{align}
f(r)&
	\leq r^n\,\cA_{\bC}(h)+C\,r^{n+\alpha}\leq (1-\eps\, \cA_{\bC}(rc)^\gamma) \frac{r}{n}\,f'(r)+ C\,r^{n+\alpha} \notag\\
	&\leq (1-\eps\, |e(r)|^\gamma) \frac{r}{n}\,f'(r)+ C\,r^{n+\alpha}\,,\notag
\end{align} 
where $e(r):=\frac{f(r)}{r^n}$. Rearranging this inequality and dividing it by $r^{n+1}$ we get
\begin{align}\label{e:decay1}
e'(r)
	&=\left(\frac{f'(r)}{r^n}-f(r)\,\frac{n}{r^{n+1}}\right)\geq n\,\eps\, \frac{e(r)^{1+\gamma}}{r(1-\eps\,|e(r)|^{\gamma})}-C \frac1{ r^{1-\alpha}}\notag\\
	&\geq n\,\eps\, \frac{e(r)^{1+\gamma}}r-C \frac1{ r^{1-\alpha}}\,.
\end{align}
We define now $\tilde e(r)= e(r)+2 \alpha^{-1} C r^\alpha$ and we notice that from the previous inequality and since $a^{1+\gamma}+b^{1+\gamma} \geq 2^{-\gamma} (a+b)^{1+\gamma}$ for any $a,b \geq 0$
\begin{equation*}
\tilde e'(r) \geq \frac{n\,\eps}{r} e(r)^{1+\gamma} + \frac{C}{r^{1-\alpha}} 
\geq \frac{n\,\eps}{r} [e(r) + {C}{r^{\frac{\alpha}{1+\gamma}}} ]^{1+\gamma}\,.
\end{equation*}
For $r$ sufficiently small, the previous inequality implies that 
\begin{equation}\label{e:decay2}
\tilde e'(r) \geq \frac{n\eps}{r} \tilde e(r)^{1+\gamma}\,.
\end{equation}
From this inequality we obtain that
$$
\frac{d}{dr}\Big(\frac{-1}{\gamma \tilde e(r)^\gamma} - n\eps \log r\Big) =\frac{1}{\tilde e(r)^{1+\gamma}}\tilde e'(r)- \frac{n\eps}{r} \geq 0
$$
and this in turn implies that $-{\tilde e(r)^{-\gamma}} - n\eps\gamma \log r$ is an increasing function of $r$, namely that $e(r)$ decays as 
$$
e(r)+2 \alpha^{-1} C r^\alpha \leq \tilde e(r) \leq ({\tilde e(r_0)^{-\gamma}+n\,\eps\, \gamma \log r_0-n\,\eps\, \gamma \log r})^{\frac {-1}{\gamma}} \leq (-n\,\eps\, \gamma \log (r/r_0))^{\frac {-1}{\gamma}}.
$$
which for $r_0$ sufficiently small implies
\begin{equation}\label{e:decay3}
e(r)\leq 2 (-n\,\eps\, \gamma \log (r/r_0))^{\frac {-1}{\gamma}} \,,\qquad s<r<r_0\,.
\end{equation}
\medskip

\noindent\textbf{Step 2: Decay of the flat norm.} Under the same assumptions as Step 1, consider the map $F(x):= \frac{x}{|x|}$ and radii $0 < s \leq r < r_0$.
By the area formula, 
\begin{align}
\mass(F_\sharp (T\res (B_r\setminus B_s))) & \leq \int_{B_r\setminus B_s}\frac{|x^\perp|}{|x|^{n+1}}\,d\|T\|
 \leq \left(\int_{B_r\setminus B_s}\frac{|x^\perp|^2}{|x|^{n+2}}\,d\|T\|\right)^{\sfrac{1}{2}}
\underbrace{\left(\int_{B_r\setminus B_s}\frac{1}{|x|^n}\,d\|T\|\right)^{\sfrac{1}{2}}}_{I_2}\notag\\
& \stackrel{\eqref{e:monotonicity formula}}{\leq} 
\left(e(r) - e(s) +C_1\, t^{\alpha_0} \right)^{\sfrac12}\,I_2.
\end{align}
We estimate $I_2$ using the graphicality of $T$ over $\bC$, that is $T\res (B_{r_0}\setminus B_{\sfrac{r_0}2})=\bG_\bC(u)$, to get
\begin{gather}
I_2^2 \leq \int_{(B_r\setminus B_s)\cap \bC} \frac1{|x|} \,Ju(x)\, dx\leq C (\log r-\log s)  \label{e:I2}
\end{gather}
In particular we conclude that
\begin{equation}\label{e:pappa}
\mass(F_\sharp (T\res (B_r\setminus B_s)))  \leq C(\log r-\log s) \left(e(r) - e(s) +C_1\, t^{\alpha_0} \right)^{\sfrac12} 
\qquad \forall\; 0 <  s \leq r < r_0,
\end{equation}
Let $0<s^{1/2}<r^{1/2}<r_0$ such that $s/r_0\in [2^{-2^{i+1}}, 2^{-2^i})$, $t/r_0\in [2^{-2^{j+1}}, 2^{-2^j})$ for some $j\leq i$ and applying the previous estimate to the exponentially dyadic decomposition, we obtain 
\begin{align}\label{e:imp_2}
\mass(F_\sharp (T\res (B_t\setminus B_s))) 
	&\leq \mass(F_\sharp (T\res (B_t\setminus B_{2^{-2^{j+1}}r_0})))  \notag\\
	&+ \mass(F_\sharp (T\res (B_{2^{-2^{i}}r_0}\setminus B_s))) +  \sum_{k=j+1}^{i-1} \mass(F_\sharp (T\res (B_{2^{-2^{k+1}}r_0}\setminus B_{2^{-2^{k}}r_0}))) \notag\\
	&\leq C \sum_{k=j}^{i} \left(\log\big(2^{-2^{k}}\big)- \log\big(2^{-2^{k+1}}\big)\right)^{1/2}\left(e\big(2^{-2^{k}}r_0\big)- e\big(2^{-2^{k+1}}r_0\big) \right)^{1/2}\notag\\
	&\leq C \sum_{k=j}^{i} 2^{k/2}e\big(2^{-2^{k}}r_0\big)^{1/2}
	\leq C \sum_{k=j}^{i} 2^{(1-1/\gamma)k/2} \notag\\
	&\leq C 2^{(1-1/\gamma)j/2} \leq C (-\log(t/r_0))^{\frac{\gamma-1}{2\gamma}}\,,
\end{align}
where $C$ is a dimensional constant that may vary from line to line.
Since $\de F_\sharp (T\res (B_r\setminus B_s)) =
\de (T_r\res B_1) - \de (T_s\res B_1)$ for a.e.~$0<s<r$, from the definition of $\mathcal{F}$ and \eqref{e:imp_2} we get
\begin{equation}\label{e:ultimo rate}
\mathcal{F}\big((T_r-T_s)\res \de B_1\big)  {\leq} C\,(-\log(r/r_0))^{\frac{\gamma-1}{2\gamma}}\,.
\end{equation}

\noindent\textbf{Step 3: Uniqueness of tangent cone.} Let $\delta=\delta(\bC)>0$ be the constant of the epiperimetric inequality, Theorem \ref{t:epi_int}, and let $\delta_1=\delta_1(\delta_0, \sfrac14, \bC)>0$ be the constant of Proposition \ref{p:sph_par} with $\tau=\sfrac14$. Moreover let $\delta_2=\delta_2(\eps_2)>0$ be the constant of Lemma \ref{l:tg_cones}. Thanks to the assumption that $\bC$ is a blow-up of $T$ at $0$, we can choose $\eps_2=\eps_2(\bC)>0$ and $r=r(\bC)>0$, with $r<\min\{\delta_2, \delta_1\}$, in such a way that
\begin{equation}\label{e:choice_const}
(\|T_{4r}\|(B_1)-\Theta_\bC)+C\,(-\log(r/r_0))^{\frac{\gamma-1}{2\gamma}}+\eps_2+\cF((T_{2r}-\bC)\res\de B_1)\leq \eta\,,
\end{equation}
where $\eta>0$ is the constant of Proposition \ref{p:sph_par}, $C$ and $\gamma$ are constants depending only on $\bC$ chosen as in \eqref{e:ultimo rate}. Notice that by Proposition \ref{p:sph_par}, the assumptions of Steps 1 and 2 are satisfied, with $t=r$ and $s=\sfrac r4$, so that by \eqref{e:ultimo rate} we get 
$$
\cF((T_r-T_{\sfrac r4})\res \de B_1)\leq C \,(-\log(r/r_0))^{\frac{\gamma-1}{2\gamma}}\,. 
$$ 
Thanks to our choice \eqref{e:choice_const} and Lemma \ref{l:tg_cones}, we can then apply Theorem \ref{t:epi_int} at the scales $[2^{-2^2}r_0,2^{-2}r_0]$, and, proceeding inductively in this way to establish \eqref{e:ultimo rate}  on exponentially dyadic scales, we conclude that the blow up is unique.
\medskip

\noindent\textbf{Step 4: proof of \eqref{e:roc1} and \eqref{e:roc2}.} The proofs of  \eqref{e:roc1} and \eqref{e:roc2} are analogous to \cite[Theorem 3.1]{DSS1} using \eqref{e:imp_2} instead of the power rate (3.13) given there.

\qed

\subsection{Proof of Corollary \ref{c:codimension1}}

We start by observing that thanks to the decomposition lemma \cite[Corollary 3.16]{Sim}, we can decompose $T= \sum_{j=-\infty}^\infty\de\a{U_j}$, with each $\de\a{U_j}$ almost area minimizing. It follows that if $\bC$ is a blow-up of $\de\a{U_J}$ at $x_0\in \supp T$, then $\bC$ is either a multiplicity-one plane or a multiplicity-one cone with $\bC\cap \de B_1$ a smooth embedded submanifold of $\de B_1$. If we can prove that each $\de\a{U_J}$ is almost area minimizing in some $\R^{n+k}$, the conclusion then follows by Theorem \ref{t:uniq}.
	
To see this it is enough to prove that $T$ is almost area minimizing on $\R^{n+k}$, where $k$ is chosen so that by Nash' theorem we can isometrically embed $N$ in $\R^{n+k}$. Indeed consider $x\in N$ and a ball $B_r (x)\subset \R^{n+k}$. If $\bar{r}$ is sufficiently small there is a well-defined $C^1$ orthogonal projection $\proj: B_{\bar{r}} (x) \to N$ with the property that $\Lip (\proj) \leq 1 + C \bA r$, where $C$ is a geometric constant and $\bA$ denotes the $L^\infty$ norm of the second fundamental form of $N$. Consider $T$ area minimizing in $N$ and assume $\bar{r}< \dist (x, \supp (\partial T))$. Let $r\leq \bar{r}$ and $S\in \bI_{n+1} (\R^{n+k})$ be such that $\supp (S)\subset B_r (x)$. We set $W:= T + \partial S$. If $\|W\| (B_r (x)) \geq \|T\| (B_r (x))$ there is nothing to prove, otherwise by the standard monotonicity formula we have 
$\|W\| (B_r (x)) \leq \|T\| (B_r (x)) \leq C r^n$. 
Then $W':= \proj_\sharp W$ is an admissible competitor for the almost minimality property of $T$ and we have 
\begin{align*}
\|T\| (B_r (x)) 
	&\leq \|W'\| (B_r (x)) +C\,r^{n+\alpha_0}\leq (\Lip (\proj))^n \|W\| (B_r (x)) 		\\
	&\leq \|W\| (B_r (x)) + C r^{n+\min\{1,\alpha_0\}}\, .
\end{align*}
\qed

\appendix

\section{Lyapunov-Schmidt reduction for the Area Functional}\label{a:LS}

We prove the following Lemma, which is adapted from \cite{Simon0}. First we need some notation; let $K := \ker \delta^2 \cA_\Sigma(0)$ and $\ell := \dim K$ which is finite by spectral theory (as $\delta^2 \cA_\Sigma(0)$ has compact resolvent). Let $P_K$ be the projection of $L^2(\Sigma; \bC^\perp)$ onto $K$ and similarly $P_{K^\perp}$ the projection onto $K^\perp$.


\begin{lemma}\label{l:LS}
There exists a neighborhood $U$ of $0$ in $C^{2,\alpha}(\Sigma; \bC^\perp)$ and an analytic map $\Upsilon: K \rightarrow K^\perp \subset C^{2}(\Sigma; \bC^\perp)$ such that 

\begin{equation}\label{e:upsilonatzero}
\Upsilon(0) = 0, \quad\mbox{and}\quad \delta \Upsilon(0) = 0,
\end{equation}
and, in addition, 
\begin{equation}\label{e:LSorth}\left\{
\begin{aligned}
&P_{K^\perp}(\delta\cA_\Sigma(\zeta+\Upsilon(\zeta)))=0, \quad \quad\quad\:\:\, \forall \zeta \in K \cap U\\
&P_{K}(\delta \cA_{\Sigma}(\zeta+ \Upsilon(\zeta))) = \nabla A(\zeta), \quad \forall \zeta \in K \cap U,
\end{aligned}\right.
\end{equation}
where $A(\zeta) = \cA(\zeta + \Upsilon(\zeta))$ for every $\zeta \in K \cap U$. Furthermore, the critical points of $\cA$ inside of $U$ are given by $$\mathcal C := \{\zeta + \Upsilon(\zeta)\mid \zeta \in U \cap K \quad \mbox{and} \quad \nabla A(\zeta) = 0\},$$ which is an analytic subvariety of the $N$-dimensional manifold given by $$\mathcal M := \{\zeta + \Upsilon(\zeta)\mid \zeta  \in U \cap K\}.$$
Finally, for all $\zeta \in U$, there is a constant $C < \infty$, such that
\begin{equation}\label{e:est_upsilon}
\|\Upsilon(P_Kc)\|\leq C\|P_Kc\|^2.
\end{equation}
\end{lemma}

\begin{proof}
Define the operator, $$\mathcal N(\zeta) := P_{K^\perp} \delta \mathcal A_\Sigma(\zeta) + P_K \zeta: L^2(\Sigma; \bC^\perp) \rightarrow L^2(\Sigma; \bC^\perp).$$ Since $\bC$ is a critical point for $\mathcal A_\Sigma$ we see that $\mathcal N(0) = 0$. Furthermore, 
$$\delta \mathcal N(0)[\zeta] = \frac{d}{dt} \mathcal N(t\zeta)|_{t= 0} = P_{K^\perp} \delta^2 \mathcal A_\Sigma(0)[\zeta, -] + P_K \zeta.$$ In particular, $\delta \mathcal N(0)$ has trivial kernel. Then Schauder estimates (applied to $-\Delta_\Sigma - B^T B - (n-1) + P_K$), imply that $\delta \mathcal N(0)$ is an isomorphism (in a neighborhood of zero) from $C^{2,\alpha}(\Sigma, \bC^\perp)$ to $C^{0,\alpha}(\Sigma, \bC^\perp)$. 

We apply the inverse function theorem to $\mathcal N$ in this neighborhood, producing the map $\Psi := \mathcal N^{-1}$ which is a bijection from a neighborhood of $0$, $W \subset C^{0,\alpha}(\Sigma; \bC^\perp)$ to $U$, a neighborhood of $0$ in $C^{2,\alpha}(\Sigma; \bC^\perp)$. 

We claim our desired map is given by $\Upsilon := P_{K^\perp}\circ \Psi: K \rightarrow K^\perp$. The first conclusion of \eqref{e:upsilonatzero} is trivial as $\Upsilon(0) = \Upsilon(\mathcal N(0)) = P_{K^\perp}(\Psi(\mathcal N(0))) = 0$. To see the second assertion we notice the more general property, that for every $\zeta \in K$ \begin{equation}\label{e:dupinthedirectionK} \delta \Upsilon(\zeta)[\eta] = (P_{K^\perp} \delta \Psi(\zeta))[\eta] = 0, \forall \eta \in K,\end{equation} by the linearity of $P_{K^\perp}$. 

To check \eqref{e:LSorth}, we first notice that \begin{equation}\label{e:zetaintermsofpsi}\zeta = \mathcal N(\Psi(\zeta)) = P_{K^\perp}\delta \cA_\Sigma(\Psi(\zeta)) + P_K\Psi(\zeta).\end{equation} Applying $P_K$ or $P_{K^\perp}$ to both sides of that equation we get $$P_{K} \zeta = P_K \Psi(\zeta) \quad\mbox{and}\quad P_{K^\perp} \zeta =  P_{K^\perp}\delta \cA(\Psi(\zeta)).$$ Plugging the first identity into the second we obtain $P_{K^\perp} \zeta = P_{K^\perp} \delta \cA(P_K \zeta + \Upsilon(\zeta)),$ which implies, for $\zeta \in K\cap U$, that $0 = P_{K^\perp} \delta \cA(\zeta + \Upsilon(\zeta)).$ 

To prove the second line of \eqref{e:LSorth}, we compute, for any $\eta \in K$; $$\begin{aligned} \left\langle \nabla A(\zeta),\eta\right\rangle = \delta \cA_\Sigma(\zeta + \Upsilon(\zeta))[\eta+ \delta \Upsilon(\zeta)[\eta]]
\stackrel{\eqref{e:dupinthedirectionK}}{=} \delta \cA_\Sigma(\zeta + \Upsilon(\zeta))[\eta],
\end{aligned}$$ 
which implies the second claim of \eqref{e:LSorth} (as $\eta \in K$ is arbitrary).

To see that all critical points are given by $\zeta  +\Upsilon(\zeta)$ we turn to \eqref{e:zetaintermsofpsi}. Let $\eta$ be an arbitrary critical point of $\cA_\Sigma$, in a neighborhood of zero. We write $\eta = \Psi(\zeta)$, and \eqref{e:zetaintermsofpsi} reads $\zeta = P_K \eta.$ Which implies $$\eta = P_K \eta + P_{K^{\perp}} \eta = \zeta + P_{K^\perp} \Psi(\zeta) = \zeta + \Upsilon(\zeta),$$ as desired (the condition on $\nabla A$ follows trivially from \eqref{e:LSorth}). 

Finally, to prove \eqref{e:est_upsilon} we must show that if $\eta \in K$ and $\mathcal N(\zeta) = \eta$ then $\|P_{K^\perp} \zeta \| \leq \|\eta\|^2$. Note that $\mathcal N(\zeta) = \eta$ implies that $\zeta = \zeta_{\perp} + \eta$, where $\zeta_{\perp}\in K^\perp$ and $\delta \cA_{\Sigma}(\zeta)[\phi] = 0$ for any $\phi \in K^\perp$. We can approximate $$\begin{aligned}\delta \cA_{\Sigma}(\zeta)[\phi] = \delta \cA_{\Sigma}(0)[\phi] + \delta^2 \cA_{\Sigma}(0)[\zeta, \phi] + O(\|\zeta\|^2)\|\phi\|\quad\text{and}\quad
\delta^2 \cA_{\Sigma}(0)[\zeta_{\perp}, \phi] = O(\|\zeta\|^2)\|\phi\|.\end{aligned}$$
Write $\zeta_{\perp} = \sum a_i \xi_i$ where $\{\xi_i\}$ is the orthonormal basis of $\delta^2 \cA_\Sigma(0)$ and where $a_i = 0$ if $\lambda_i = 0$. Letting $\phi = \sum_{\{\lambda_i > 0\}} a_i \xi_i - \sum_{\{\lambda_i < 0\}} a_i \xi_i$ above yields \eqref{e:est_upsilon}.
\end{proof}

\bibliographystyle{plain}
\bibliography{references-Cal}

\end{document}